\documentclass[preprint,12pt]{elsarticle}
\usepackage[colorlinks=true]{hyperref} 

\usepackage{amssymb}
\usepackage{amsmath}
\usepackage{amsthm}

\usepackage[all]{nowidow}

\usepackage{lineno}

\journal{~}
\newtheorem{theorem}{Theorem}
\newtheorem{definition}{Definition}
\newtheorem{lemma}{Lemma}
\newtheorem{remark}{Remark}
\newtheorem{observation}{Observation} 

\newtheorem{corollary}{Corollary}

\usepackage{amsmath}
\usepackage{amssymb}
\usepackage{amsthm}
\usepackage{mleftright}
\usepackage{enumerate}
\usepackage[utf8]{inputenc}

\usepackage{tikz}
\usetikzlibrary{backgrounds}
\pgfdeclarelayer{bg}
\pgfsetlayers{bg,main}
\usetikzlibrary{decorations.pathreplacing,decorations.markings}
\definecolor{myhamcyclecolor}{HTML}{000000}
\definecolor{mystdegdecolor}{HTML}{999999}
 
\makeatletter
\def\ps@pprintTitle{
  \let\@oddhead\@empty
  \let\@evenhead\@empty
  \let\@oddfoot\@empty
  \let\@evenfoot\@oddfoot
}

\begin{document}

\begin{frontmatter}

	\title{Putting Tutte's counterexample to Tait's conjecture in perspective to hamiltonicity and non-hamiltonicity in certain planar cubic graphs}

	\author[label1]{Herbert Fleischner}
	\author[label1]{Enrico Iurlano}
	\author[label1]{Günther R.~Raidl\corref{cor1}} \cortext[cor1]{Corresponding author}
	\affiliation[label1]{organization={Algorithms and Complexity Group, TU Wien},city={Vienna},
		country={Austria}}

	\begin{abstract}
		Using the graphs of prisms and Tutte Fragments, we construct an infinite family of hamiltonian and non-hamiltonian graphs in which Tutte's counterexample to Tait's conjecture appears in a certain sense as a minimal element.
		We observe that generalizations of the minimum-cardinality counterexamples of Holton and McKay to Tait's conjecture are as well contained in this family.
	\end{abstract}


\end{frontmatter}

\nocite{bagheri2021hamiltonian,bagheri2022finding,bondy1976graph,cada2004hamiltonian,fleischner1988prism,holton1988smallest,tutte1946hamiltonian}

\section{Introduction and preliminary considerations}\label{sec:notation-and-preliminaries}

The \emph{Tutte Fragment} (TF) is a particular planar subcubic graph with only three $2$-valent vertices~\cite{tutte1946hamiltonian}.
Figure~\ref{fig:isolated-tutte-fragment} illustrates it by placing the three $2$-valent vertices $a$, $b$, and $c$ on the extreme points of an equilateral triangle.
The TF has the property that no hamiltonian paths exists with endpoints $a$ and $b$, but there are hamiltonian paths from $a$ to $c$ and $b$ to $c$.
In $1946$, W.~T.~Tutte~\cite{tutte1946hamiltonian} constructed a $3$-connected planar cubic graph on $46$ vertices admitting no hamiltonian cycle by suitably linking copies of the TF via their $2$-valent vertices.
This counterexample disproved for the first time the longstanding conjecture of Tait from $1884$, claiming that every $3$-regular, $3$-connected, planar graph is hamiltonian.

Note, however, that D.~W.~Barnette proposed in 1969 the conjecture that every planar, $3$-connected, cubic bipartite graph is hamiltonian.
So far, there is only one major partial solution of this conjecture; see \cite{bagheri2021hamiltonian}.
Calling such graphs \emph{Barnette graphs} it was shown there that if two of the color classes of a Barnette graph consist of quadrilaterals and hexagons only, then it is hamiltonian (in fact, the main result there is even somewhat more general).
Moreover, it was shown in~\cite{bagheri2022finding} that if Barnette's conjecture is false, then the decision problem whether a Barnette graph is hamiltonian, is NP-complete.
The paper~\cite{cada2004hamiltonian} studies decompositions into hamiltonian cycles of prisms over $3$-connected planar bipartite cubic graphs; furthermore, it shows that prisms over any $3$-connected cubic graph are hamiltonian.
Prior to~\cite{cada2004hamiltonian} it had been proved by Fleischner,
without invoking the Four Color Theorem, that the prism of a 2-connected
planar cubic graphs is hamiltonian~\cite{fleischner1988prism}.

The aim of the current work is to show that Tutte's counterexample but also the counterexample of Holton and McKay to Tait's conjecture are not isolated incidences.
Rather, they are in a way minimal examples in an infinite family of counterexamples to Tait's conjecture.
Moreover, the same TF-inflations (see below for this concept) applied to even-sided prisms yield hamiltonian graphs (Theorem~\ref{thm:main-theorem-equivalence-existence-parity} and Corollary~\ref{cor:inflated-prism-two-pillars-case}).

\bigskip

In our considerations all appearing graphs are assumed to be simple and undirected.
For standard terminology of graph theory we refer to the textbook by Bondy and Murty~\cite{bondy1976graph}.
The \textit{$n$-prism} $C_{n}\,\square\,K_2$ is the Cartesian product of $C_n$, the cycle of length $n\geq 3$, and $K_2$, the path on two vertices.
For simplicity, assume the following classification of the vertices: Let the vertex set be partitioned into vertices of a \emph{base cycle} $\{s_i: i\in\mathbb{Z}_n\}$ and vertices of a \emph{top cycle} $\{t_i: i\in\mathbb{Z}_n\}$; accordingly, the set of edges is given by the union of both cycles' edges $\{{s_i}{s_{i+1}}, {t_i}{t_{i+1}} : i\in\mathbb{Z}_n\}$ and the set of \emph{pillars} $\{{s_i}{t_i}:i\in\mathbb{Z}_{n}\}$.

A \emph{triangle inflation} (see also~\cite[p.~51]{cada2004hamiltonian}) replaces a $3$-valent vertex $v$ with a triangle $\Delta$ having vertex set $V(\Delta)=\{v^a, v^b, v^c\}$ such that each of the triangle's vertices is connected to a unique former neighbor of $v$.
Just as one speaks of inflating a $3$-valent vertex into a triangle, we will speak of \emph{TF-inflations}.
\begin{definition}[TF-inflation]
	A TF-inflation of a $3$-valent vertex $v$ is a triangle inflation of $v$ followed by the following operations:
	The edges of $\Delta$ are deleted and each $v^x$, $x\in\{a,b,c\}$, is identified with a unique corresponding $2$-valent vertex $x$ of the TF.
	The resulting graph (having $14$ additional vertices) is $3$-regular; it is said to derive from $C_{n}\,\square\,K_2$ by TF-inflation.
\end{definition}
Given the asymmetry of the TF, depending on which of the former $v$-incident edges are associated to $v^a$, $v^b$, or $v^c$ of the TF, we obtain six possibilities for a TF-inflation.
However, in the case of prisms, for every TF-inflation we insist that $v^c = c$ is incident to the corresponding pillar, leaving two possibilities for a TF-inflation in $n$-prisms.
After TF-inflations we regard all edges of the prism contained in $\{{s_i}{t_i}, {s_i^c}{t_i}, {t_i^c}{s_i}, {t_i^c}{s_i^c} :i\in\mathbb{Z}_{n}\}$ as \emph{pillars of the TF-inflated prism}.
TF-inflations on a prism are illustrated in Fig.~\ref{fig:tf-inflation-prism-multiple-pillars}.
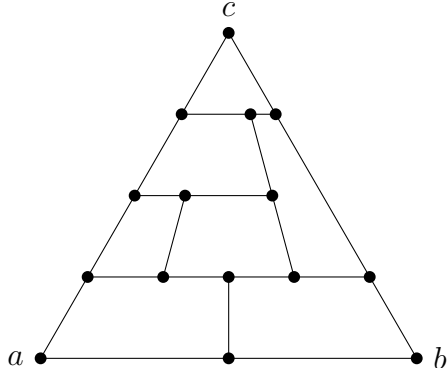
\begin{figure}[tb!]
	\centering
	\begin{tikzpicture}[line cap=round,line join=round,x=0.5cm,y=0.5cm,bv/.style={
					circle,
					draw=black,
					fill=black,
					thick,
					inner sep=1.3pt,
					minimum size=0.09cm
				}]
		\clip(-1,-0.25) rectangle (11,9.76);
		\draw (0,0)-- (5,8.660254037844386);
		\draw (10,0)-- (5,8.660254037844386);
		\draw (0,0)-- (10,0);
		\draw (1.25,2.1650635094610955)-- (8.75,2.165063509461097);
		\draw (2.5,4.330127018922193)-- (6.160254037844386,4.330127018922193);
		\draw (3.75,6.495190528383288)-- (6.25,6.495190528383288);
		\draw (5.580127018922194,6.495190528383287)-- (6.74038105676658,2.165063509461097);
		\draw (3.839745962155613,4.330127018922193)-- (3.25961894323342,2.1650635094610955);
		\draw (5,2.1650635094610964)-- (5,0);

		\node at (10,0) [bv,label=right:{$b$}]{};
		\node at (0,0) [bv,label=left:{$a$}]{};
		\node at (5,8.660254037844386) [bv,label=above:{$c$}]{};
		\node at (5,2.1650635094610964) [bv]{};
		\node at (1.25,2.1650635094610955) [bv]{};
		\node at (2.5,4.330127018922193) [bv]{};
		\node at (3.75,6.495190528383288) [bv]{};
		\node at (8.75,2.165063509461097) [bv]{};
		\node at (6.25,6.495190528383288) [bv]{};
		\node at (3.25961894323342,2.1650635094610955) [bv]{};
		\node at (6.74038105676658,2.165063509461097) [bv]{};
		\node at (3.839745962155613,4.330127018922193) [bv]{};
		\node at (6.160254037844386,4.330127018922193) [bv]{};
		\node at (5.580127018922194,6.495190528383287) [bv]{};
		\node at (5,0) [bv]{};
	\end{tikzpicture}

	\caption{Celebrated Tutte fragment (TF) found in $1946$.
		No hamiltonian path of the TF with endpoints $a$ and $b$ exists.
		However, certain hamiltonian paths with endpoint-pair $(a,c)$ as well as $(b,c)$ exist.}\label{fig:isolated-tutte-fragment}
\end{figure}
\begin{figure}[tb!]
	\centering
	\includegraphics[scale=1]{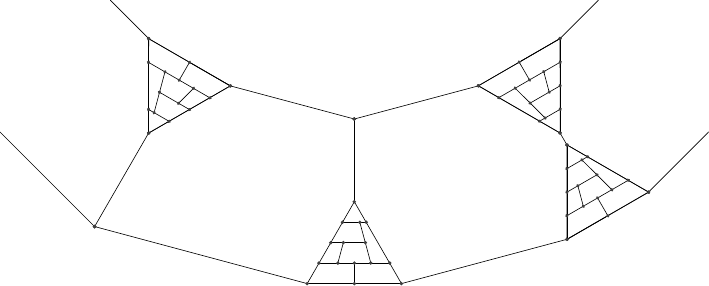}

	\caption{Three pillars of an $n$-prism with single- and both-sided TF-inflations:
		A single TF-inflation at the pillar's top (respectively bottom) vertex is shown for the left (respectively middle) pillar.
		In case of the right pillar, both vertices have been subjected to a TF-inflation.
	}\label{fig:tf-inflation-prism-multiple-pillars}
\end{figure}
By exhaustion one can derive the following result including the hamiltonian paths in question.
\begin{lemma}[\cite{tutte1946hamiltonian}~Hamiltonian paths of the TF]\label{lem:hamiltonian-paths-of-the-tf}
	Let TF be the graph displayed in Fig.~\ref{fig:isolated-tutte-fragment} with correspondingly assigned labels $a$, $b$, and $c$ for the $2$-valent vertices.
	Then, no hamiltonian path of the TF with endpoints $a$ and $b$ exists.
	However, two hamiltonian paths with endpoint-pair $(a,c)$, and four hamiltonian paths with endpoint-pair $(b,c)$ exist.
\end{lemma}

\section{Results}\label{sec:results}

We start with some preparatory observations.
\begin{lemma}\label{lem:allowed-numbers-of-used-pillars}
	The number $r\in\mathbb{N}$ of pillars contained in a hamiltonian cycle of $C_{n}\,\square\,K_2$ can only be $r\in\{2,n\}$.
	If $r=2$, there must be a prism's quadrangle containing both pillars.
\end{lemma}
\begin{proof}
	Requiring a cycle to include less than two pillars does not result in a hamiltonian cycle.
	If more than two but less than all pillars are contained in a hamiltonian cycle $H$ of $C_{n}\,\square\,K_2$, one can find distinct pillars $p$, $p'$ and $p''$ such that, firstly, all pillars $u_1,\dots,u_k$ (with $k\geq 1$) between $p$ and $p'$ are uncovered by $H$; and, secondly, a nonnegative number of uncovered pillars $u'_1,\dots,u'_{k'}$ ($k'\geq 0$) lies between $p'$ and $p''$; see Fig.~\ref{fig:uncovered-pillars}.
	Without loss of generality, we can assume that the part of $H$ being incident to $u_1,\dots,u_k$ passes through the base cycle of $C_{n}\,\square\,K_2$.
	Consequently, the top vertices of $u_1,\dots,u_k$ necessarily would remain uncovered by $H$, contradicting hamiltonicity.

	The argument is analogously applicable to the case that just two distinct pillars $p$ and $p'$, which are not part of the same prism's quadrangle, are covered by $H$.
\end{proof}
\begin{figure}[tb!]
	\centering
	\begin{tikzpicture}[
			thick,
			bv/.style={
					circle,
					draw=black,
					fill=black,
					thick,
					inner sep=1.5pt,
					minimum size=0.09cm
				},
			x=0.5cm,y=0.5cm
		]

		\node at (0,0) [bv]{};
		\node at (0,4) [bv]{};

		\node at (2,0) [bv]{};
		\node at (2,4) [bv]{};
		\node at (4,0) [bv]{};
		\node at (4,4) [bv]{};
		\node at (6,0) [bv]{};
		\node at (6,4) [bv]{};
		\node at (8,0) [bv]{};
		\node at (8,4) [bv]{};

		\node at (10,0) [bv]{};
		\node at (10,4) [bv]{};
		\node at (12,0) [bv]{};
		\node at (12,4) [bv]{};
		\node at (14,0) [bv]{};
		\node at (14,4) [bv]{};

		\node at (16,0) [bv]{};
		\node at (16,4) [bv]{};

		\begin{pgfonlayer}{bg}
			\draw[line width=2.0000000pt,color=myhamcyclecolor] (0,0) -- (0,4);
			\draw[line width=.5pt,color=mystdegdecolor] (2,0) -- (2,4);
			\draw[line width=.5pt,color=mystdegdecolor] (4,0) -- (4,4);
			\draw[line width=2.0000000pt,dotted,color=myhamcyclecolor] (2,0) -- (4,0);
			\draw[line width=.5pt,dotted,color=mystdegdecolor] (2,4) -- (4,4);

			\draw[line width=2.0000000pt,dotted,color=myhamcyclecolor] (4,0) -- (6,0);
			\draw[line width=.5pt,dotted,color=mystdegdecolor] (4,4) -- (6,4);

			\draw[line width=2.0000000pt,color=myhamcyclecolor] (8,0) -- (8,4);

			\draw[line width=.5pt,color=mystdegdecolor] (6,0) -- (6,4);

			\draw[line width=2.0000000pt,color=myhamcyclecolor] (0,0) -- (2,0);
			\draw[line width=.5pt,color=mystdegdecolor] (0,4) -- (2,4);

			\draw[line width=2.0000000pt,color=myhamcyclecolor] (6,0) -- (8,0);
			\draw[line width=.5pt,color=mystdegdecolor] (6,4) -- (8,4);

			\draw[line width=.5pt,color=mystdegdecolor] (8,0) -- (10,0);
			\draw[line width=2.0000000pt,color=myhamcyclecolor] (8,4) -- (10,4);

			\draw[line width=.5pt,dotted,color=mystdegdecolor] (10,0) -- (12,0);
			\draw[line width=2.0000000pt,dotted,color=myhamcyclecolor] (10,4) -- (12,4);

			\draw[line width=.5pt,dotted] (12,0) -- (14,0);
			\draw[line width=2.0000000pt,dotted,color=myhamcyclecolor] (12,4) -- (14,4);

			\draw[line width=.5pt,color=mystdegdecolor] (10,0) -- (10,4);
			\draw[line width=.5pt,color=mystdegdecolor] (12,0) -- (12,4);
			\draw[line width=.5pt,color=mystdegdecolor] (14,0) -- (14,4);

			\draw[line width=.5pt,color=mystdegdecolor] (14,0) -- (16,0);
			\draw[line width=2.0000000pt,color=myhamcyclecolor] (14,4) -- (16,4);

			\draw[line width=2.0000000pt,color=myhamcyclecolor] (16,0) -- (16,4);
		\end{pgfonlayer}

		\draw [decorate,decoration={brace,amplitude=5pt,mirror,raise=0.5em}]
		(2,0) -- (6,0) node[midway,yshift=-2.7em,xshift=0em]{\centering \parbox{2cm}{$u_1,\dots,u_k$\\\centering $k\geq 1$}};

		\draw [decorate,decoration={brace,amplitude=5pt,mirror,raise=0.5em}]
		(10,0) -- (14,0) node[midway,yshift=-2.7em,xshift=0em]{\centering \parbox{2cm}{$u'_1,\dots,u'_{k'}$\\\centering $k'\geq 0$}};

		\node at (0,-2.15em) []{$p$};
		\node at (8,-2.15em) []{$p'$};
		\node at (16,-2.15em) []{$p''$};
	\end{tikzpicture}
	\caption{Each hamiltonian cycle of $C_{n}\,\square\,K_2$ either covers exactly two consecutive or all pillars. Pillars that are uncovered in a given hamiltonian cycle are labeled~$u_i$,~$u'_i$.}\label{fig:uncovered-pillars}
\end{figure}
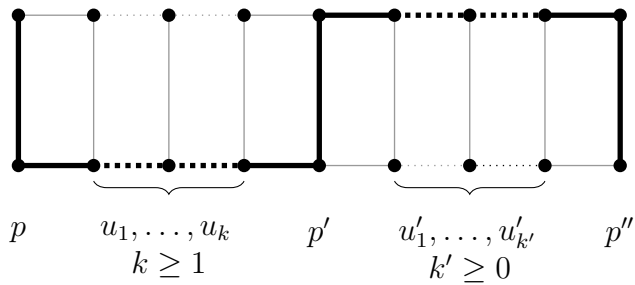
\begin{definition}
	A hamiltonian cycle $H$ in a prism $C_{n}\,\square\,K_2$ is called \emph{meandering} if no two edges of the top or bottom cycle are adjacent in $H$.
\end{definition}
It follows that a meandering hamiltonian cycle in $C_{n}\,\square\,K_2$ must contain all pillars of $C_{n}\,\square\,K_2$.
\begin{lemma}\label{lem:parity-implies-existence-nonexistence}
	The following is true.
	\begin{enumerate}[(i)]
		\item If $n$ is odd, there is no hamiltonian cycle of $C_{n}\,\square\,K_2$ containing all pillars.\label{ite:lem:parity-implies-existence-nonexistence-odd-case}
		\item If $n$ is even, there are precisely two cycles being hamiltonian and covering all pillars of $C_{n}\,\square\,K_2$; both are \emph{meandering}.\label{ite:lem:parity-implies-existence-nonexistence-even-case}
	\end{enumerate}
\end{lemma}
\begin{proof}
	Let $B$ denote the base cycle and $T$ the top cycle of $C_{n}\,\square\,K_2$.
	Let $\Pi$ denote the set of pillars of $C_{n}\,\square\,K_2$.
	Let $C=(e_1,\dots,e_{2_n})$ denote the sequence of edges of a hamiltonian cycle of $C_{n}\,\square\,K_2$ further assuming that $e_1 = {s_0}{t_0}$.
	There can be just two scenarios meeting our assumptions: either $(e_{2n},e_2)\in B\times T$, or $(e_{2n},e_2)\in T\times B$.
	Both scenarios propagate a meandering form $C\in \Pi \times T \times \Pi \times B \times \dots \times \Pi \times T \times \Pi \times B\times \cdots$---with swapped roles of $B$ and $T$ in case of the initial assumption $(e_{2n},e_2)\in T\times B$.
	Only when $n$ is even, such a meandering hamiltonian cycle can successfully be closed; according to the two possible initial assumptions, two distinct meandering hamiltonian cycles then exist.
\end{proof}
\begin{remark}\label{rem:two-or-less-pillars-case-is-trivial}
	In contrast, it is easy to verify that a hamiltonian cycle of $C_{n}\,\square\,K_2$ containing only two pillars shared by a quadrangle exist for arbitrary $n$:
	It consists of the two pillars and all edges of the top and bottom cycles except the two edges joining the pillars in $C_{n}\,\square\,K_2$; this hamiltonian cycle does not meander like in the other cases.
\end{remark}
\begin{observation}\label{obs:usage-of-pillars-due-to-inflation}
	Due to the structure of the TF (recall Lemma~\ref{lem:hamiltonian-paths-of-the-tf}), whenever a vertex of $C_{n}\,\square\,K_2$ has been TF-inflated, each hamiltonian cycle of the inflated version of $C_{n}\,\square\,K_2$ needs to contain the incident pillar.
\end{observation}

\begin{theorem}\label{thm:main-theorem-equivalence-existence-parity}
	Suppose $G$ is obtained from $C_{n}\,\square\,K_2$ by applying a TF-inflation to at least one vertex of $r\geq 3$ distinct pillars of $C_{n}\,\square\,K_2$.
	Then, the following assertions hold.
	\begin{enumerate}[(i)]
		\item Graph $G$ is hamiltonian if and only if $n$ is even.
		      Furthermore, for even $n$, every hamiltonian cycle of $G$ naturally corresponds to one of the two meandering hamiltonian cycles of $C_{n}\,\square\,K_2$.\label{ite:thm:main-theorem-equivalence-existence-parity-even-case}
		\item If $n$ is odd, then $G$ has hamiltonian paths.\label{ite:thm:main-theorem-equivalence-existence-parity-odd-case}
	\end{enumerate}
\end{theorem}
\begin{proof}
	Let us first prove~\eqref{ite:thm:main-theorem-equivalence-existence-parity-even-case}.
	By contrapositive, if $n$ is odd and $G$ has a hamiltonian cycle $H$, then the cycle $H'$ resulting from the contraction of those edges in $H$ that stem from a TF-inflation, would yield a hamiltonian cycle for $C_{n}\,\square\,K_2$.
	From the assumption combined with Observation~\ref{obs:usage-of-pillars-due-to-inflation} we get that $H'$ necessarily covers at least three pillars of $C_{n}\,\square\,K_2$, implying by Lemma~\ref{lem:allowed-numbers-of-used-pillars} that all pillars must be covered by $H'$.
	However, since $n$ is odd, this yields a contradiction to Lemma~\ref{lem:parity-implies-existence-nonexistence}.
	Consequently no hamiltonian cycle $H$ in $G$ exists.

	Let us now prove that for even $n$, hamiltonicity of $G$ ensues.
	By Lem\-ma~\ref{lem:parity-implies-existence-nonexistence}\,\eqref{ite:lem:parity-implies-existence-nonexistence-even-case}, there is a hamiltonian cycle $H$ covering all pillars of $C_{n}\,\square\,K_2$.
	Whenever a TF-inflation is applied to a vertex $v$ of $C_{n}\,\square\,K_2$, we can augment $H$ towards a feasible hamiltonian cycle for the post-inflation version of $C_{n}\,\square\,K_2$:
	Without loss of generality we can assume $v=s_i$, ${s_i}{s_{i+1}}\in H$, and adjacency of $s_i^{b}$ and $s_{i+1}$.
	Then one simply replaces the former pillar ${s_i}{t_i}$ by the union of the new pillar ${s_i^c}{t_i}$ and the hamiltonian path connecting $s_i^c$ with $s_i^b$ (running within the TF and existing due to Lemma~\ref{lem:hamiltonian-paths-of-the-tf}).
	As these inter-TF hamiltonian paths are not unique (see Lemma~\ref{lem:hamiltonian-paths-of-the-tf}), multiple versions for the augmented hamiltonian cycle are conceivable; see Section~\ref{sec:counting-hamilton-cycles} below.
	However, they all derive from the same hamiltonian cycle in $C_{n}\,\square\,K_2$.
	The argument applies for iterated TF-inflation; consequently, $G$ is hamiltonian.

	Let us show that an existing hamiltonian path $H$ of $G$ for even $n$ is---after contracting its inter-TF edges---coincident with a meandering path of $C_{n}\,\square\,K_2$.
	The number of pillars covered by $H$ in $G$ is invariant under inter-TF paths contractions.
	By assumption, at least three pillars of $G$ are covered by $H$ implying that at least three pillars in $C_{n}\,\square\,K_2$ are covered by the contracted cycle $H'$.
	By Lemma~\ref{lem:allowed-numbers-of-used-pillars}, the latter insight implies that $n$ pillars must be covered by $H'$, in turn meaning by Lemma~\ref{lem:parity-implies-existence-nonexistence}\,\eqref{ite:lem:parity-implies-existence-nonexistence-even-case} that $H'$ is a meandering hamiltonian cycle.

	\medskip

	Next, let us show~\eqref{ite:thm:main-theorem-equivalence-existence-parity-odd-case}.
	Insert in $G$ an additional pillar which is not subject to any TF-inflation.
	This augmented graph equivalently derives from the even-sided prism $C_{n+1}\,\square\,K_2$ by suitable TF-inflations and therefore possesses a hamiltonian cycle $H_\ast$ containing all pillars.
	By excluding from $H_\ast$ this additional pillar and its two incident edges in $H$, we obtain a hamiltonian path in $G$.
\end{proof}
From the case $r=2$ in Lemma~\ref{lem:allowed-numbers-of-used-pillars} we can analogously derive the following result.
\begin{corollary}\label{cor:inflated-prism-two-pillars-case}
	The conclusions of Theorem~\ref{thm:main-theorem-equivalence-existence-parity} remain valid if $r=2$ provided no quadrangle of the prism contains both pillars in question.
\end{corollary}
We observe that the minimum-cardinality counterexamples to Tait's conjecture due to Holton and McKay~\cite[Figure 1.1]{holton1988smallest} are special cases more generally captured by Corollary~\ref{cor:inflated-prism-two-pillars-case}.
We note in passing that in the case of $r=2$ where the two pillars in question are part of a quadrangle of the prism, the inflated graph has an associated unique hamiltonian cycle in $C_{n}\,\square\,K_2$; see Remark~\ref{rem:two-or-less-pillars-case-is-trivial}.

\subsection{Counting hamiltonian cycles after TF-inflation} \label{sec:counting-hamilton-cycles}

Fixing a prism $C_{n}\,\square\,K_2$, for even $n\geq 4$, we focus on its two meandering hamiltonian cycles $H_1$ and $H_2$.
Let $F$ be a TF used in a TF-inflation of $C_{n}\,\square\,K_2$.
There are two hamiltonian paths in $F$ starting with $a$ and ending in $c$. Likewise, there are four hamiltonian paths in $F$ starting with $b$ and ending in $c$.
Correspondingly we speak of an $[a,c]$-traversal, respectively of a $[b,c]$-traversal, of $F$.
Now, if in the TF-inflation of $C_{n}\,\square\,K_2$ the hamiltonian cycle $H_1$ is being expanded by an $[a,c]$-traversal of $F$, then the hamiltonian cycle $H_2$ is being expanded by a $[b,c]$-traversal of $F$.
Suppose $G$ derives from $C_{n}\,\square\,K_2$ by $q$ TF-inflations.
Then, we have the following:
If $H_1$ expands in $q_1$ copies of $F$ by an $[a,c]$-traversal, then it expands in the remaining $q-q_1$ copies of $F$ by a $[b,c]$-traversal.
Correspondingly, $H_2$ expands in $q_1$ copies of $F$ by a $[b,c]$-traversal, and it expands in the remaining $q-q_1$ copies of $F$ by an $[a,c]$-traversal.
It follows that $H_1$ yields $2^{q_1} \cdot 4^{q-q_1}$ hamiltonian cycles in $G$, and $H_2$ yields $4^{q_1} \cdot 2^{q-q_1}$ hamiltonian cycles in $G$.
This gives a total number of
\begin{equation*}
	h(q, q_1) = 2^{q_1} \cdot 4^{q-q_1}+4^{q_1} \cdot 2^{q-q_1}
\end{equation*}
hamiltonian cycles.
Using the arithmetic-geometric mean-inequality (respectively, a trivial upper bound) we get
\begin{equation*}
	2^{3q/2+1} = 2\sqrt{2^q4^q}\leq h(q,q_1)\leq 2^{2q+1}.
\end{equation*}

\subsection{Recovering Tutte's counterexample}

Having a closer look at Theorem~\ref{thm:main-theorem-equivalence-existence-parity}, it turns out that the latter examines---up to a subsequently described contraction---generalizations of Tutte's counterexample to $n$-prisms for arbitrary odd $n$.
Let us apply Theorem~\ref{thm:main-theorem-equivalence-existence-parity} to the prism $C_{3}\,\square\,K_2$ specifically, subject to the following conditions.
\begin{enumerate}[(i)]
	\item The three TF-inflations are performed at the three vertices of the base cycle $(s_0,s_1,s_2)$.
	\item The resulting graph $G^\ast$ should be as symmetrical as possible; this means that without loss of generality $G^\ast$ should contain in particular the edges ${s_0^b}{s_1^a}$, ${s_1^b}{s_2^a}$, and ${s_2^b}{s_0^a}$.
\end{enumerate}
In fact, the symmetry group of $G^\ast$ is the cyclic group of order three, and it is fixed-point-free.
However, the top cycle is still a triangle $\Delta$.
By contracting $\Delta$, the resulting graph or its mirror image is the usual presentation of Tutte's counterexample to Tait's conjecture.

\section{Conclusion}\label{sec:conclusion}

Combining the $n$-sided prism with the operation of TF-inflation, we created an infinite family of hamiltonian planar $3$-connected cubic graphs (Theorem~\ref{thm:main-theorem-equivalence-existence-parity} and Corollary~\ref{cor:inflated-prism-two-pillars-case}, case $n$ even), respectively such non-hamiltonian graphs (corresponding odd case).
This was done subject to the condition that the $2$-valent vertex $c$ has to be incident with the corresponding pillar.
Consequently, Tutte's counterexample to Tait's conjecture resulted from a minimal element (in the case $r>2$) after contracting the top triangle, whereas Holton and McKay's counterexample to Tait's conjecture appears as a minimal element in the case $r = 2$; see Corollary~\ref{cor:inflated-prism-two-pillars-case}.
In other words, our results show that these counterexamples are no coincidence; rather, they are in a certain sense just minimal elements in the corresponding partial order of counterexamples.
On top of this, we also showed how divergent the number of hamiltonian cycles may be depending on the applied TF-inflations.

\section*{Acknowledgments}

We want to thank the first author of this paper, Herbert Fleischner, who unfortunately passed away just few days before the submission of this work.
The fundamental ideas of this paper are his as well as much of the writing.
Herbert, we are grateful to have had the opportunity to work together with you, it was always a big pleasure.
We will miss you!

This work is partially supported by the program VGSCO [10.55776/W1260-N35] of the Austrian Science Fund (FWF).

\bibliographystyle{elsarticle-num}

\end{document}